\documentclass[11pt,oneside,reqno]{amsart}
\usepackage[margin=0.9in]{geometry}

\usepackage{amsmath,amssymb,amsthm}
\usepackage{graphicx}
\usepackage[expansion=false]{microtype}
\usepackage[hidelinks]{hyperref}

\graphicspath{{figures/}}

\theoremstyle{plain}
\newtheorem{theorem}{Theorem}
\newtheorem{proposition}[theorem]{Proposition}
\newtheorem{lemma}[theorem]{Lemma}
\newtheorem{corollary}[theorem]{Corollary}
\theoremstyle{definition}
\newtheorem{remark}[theorem]{Remark}
\newtheorem{question}[theorem]{Question}

\DeclareMathOperator{\dist}{dist}
\DeclareMathOperator{\area}{area}
\DeclareMathOperator{\length}{length}
\newcommand{\Sph}{\mathbb{S}^{2}}
\newcommand{\R}{\mathbb{R}}
\newcommand{\bta}[2]{\beta^{#1,#2}}

\title[Mean distance to a curve on the sphere]
{The mean distance to a simple closed curve on the sphere}

\author{J. V. M. Pimentel}

\date{\today}


\begin{document}

\begin{abstract}
Kimberling's Problem~10 asks for a simple closed curve of prescribed length
$L$ (in particular, $L=4\pi$) on the unit sphere
minimizing the mean geodesic distance $\mathcal{J}$ from a point of the sphere to the
curve. For a positive integer $n$, put $\vartheta_{n}=\pi/(2n)$ and
$L_{n}=2\pi/\sin\vartheta_{n}$. We show
that the minimum of $\mathcal{J}$ over rectifiable simple closed curves of length at most
$L_{n}$ equals $\vartheta_{n}-\tan(\vartheta_{n}/2)$, that it is attained only by curves of length
exactly $L_{n}$, and that the sphere-filling ropes $\bta{n}{k}$ of Gerlach and
von der Mosel attain it. Kimberling's case is $n=3$: at $L=4\pi$ the minimum is
$\pi/6+\sqrt{3}-2=0.255649\ldots$, attained by an explicit six-arc curve and by
its mirror image. For $L\le2\pi$ we determine $J(L)$, the infimum of
$\mathcal{J}$ over curves of length $L$, exactly: it equals $\pi/2-L/(2\pi)$,
attained precisely by the circles of length $L$. At the lengths $L_{n}$ we do not
classify all minimizers, but show that every one of
them bisects the sphere into two disks of area $2\pi$ and inradius
$\vartheta_{n}$ whose inward collars have the largest possible area at every
depth. The great circle is the only minimizer for $n=1$, and the $\bta{n}{k}$
are, up to congruence, the only ones of thickness at least $\sin\vartheta_{n}$.
For arbitrary $L$ the function $J$ is nonincreasing, and
together with the above this brackets it between two explicit values.
\end{abstract}

\maketitle

\section{Introduction}

Kimberling's Problem~10 \cite{Kimberling} reads as follows.
{\tiny
\begin{quote}
Let $S$ be given by $x^{2}+y^{2}+z^{2}=1$. Suppose that $C$ is a simple closed
curve on $S$. For any point $P$ on $S$, define the distance from $P$ to $C$ to
be the minimal arclength from $P$ to $Q$ as $Q$ goes around $C$. Find
parametric equations for such a curve $C$ of length $4\pi$ that minimizes the
mean distance from $S$ to $C$; that is, the mean distance taken over all points
$P$ on $S$. The solution must include proof of minimization. Can you solve this
problem with arbitrary $L>2\pi$ in place of $4\pi$?
\end{quote}
}
\noindent
Write $\Sph$ for the unit sphere. Throughout, \emph{curve} means a
rectifiable simple closed curve $\mathcal{C}\subset\Sph$, and $L$ denotes its
length. Define
\begin{equation}\label{eq:Jcal}
  \mathcal{J}(\mathcal{C})=\frac{1}{4\pi}\int_{\Sph}\dist(x,\mathcal{C})\,dA(x),
  \qquad
  \dist(x,\mathcal{C})=\min_{y\in \mathcal{C}}\arccos\langle x,y\rangle ,
\end{equation}
with $dA$ the area element of $\Sph$; we write $\dist(x,y)=\arccos\langle
x,y\rangle$ for the geodesic distance between two points as well. Put
\begin{equation}
  J(L)=\inf\bigl\{\mathcal{J}(\mathcal{C}):\ \length(\mathcal{C})=L\bigr\}.
\end{equation}

\begin{proposition}\label{prop:short}
Every rectifiable closed curve $\mathcal{C}\subset\Sph$ of length $L$, simple or
not, satisfies
\begin{equation}\label{eq:short}
  \mathcal{J}(\mathcal{C})\ \ge\ \frac{\pi}{2}-\frac{L}{2\pi} .
\end{equation}
For $L\le2\pi$ the circles of length $L$ attain it, and no other simple closed
curve does, so
\begin{equation}
  J(L)=\frac{\pi}{2}-\frac{L}{2\pi},\qquad 0<L\le2\pi .
\end{equation}
\end{proposition}

Proposition~\ref{prop:short} is proved in Section~\ref{sec:short}.
Only $L>2\pi$ remains; we work throughout with $L\ge2\pi$, keeping the endpoint
because it is the first of the lengths $L_{n}$ introduced below. Attached to such
a length is the radius
$\vartheta=\vartheta(L)\in(0,\pi/2]$ determined by
\begin{equation}\label{eq:vartheta}
  \sin\vartheta=\frac{2\pi}{L},\qquad\text{equivalently}\qquad
  L\sin\vartheta=2\pi .
\end{equation}
Corollary~\ref{cor:cover} will show that $\vartheta$ is the smallest radius at
which a tube about a curve of length $L$ can cover the sphere.
The quantity to be minimized will be bounded below by
\begin{equation}\label{eq:j}
  j(L)=\vartheta-\tan\frac{\vartheta}{2}
  =\arcsin\frac{2\pi}{L}-\frac{1-\sqrt{1-(2\pi/L)^{2}}}{2\pi/L} .
\end{equation}
Here $\vartheta$ is strictly decreasing in $L$, while
$\theta\mapsto\theta-\tan\frac{\theta}{2}$ is strictly increasing on $[0,\pi/2]$,
so $j$ is strictly decreasing on $[2\pi,\infty)$. Put
\[
  \vartheta_{n}=\frac{\pi}{2n},\qquad L_{n}=\frac{2\pi}{\sin\vartheta_{n}}
  \quad(n=1,2,3,\dots),
\]
so that $\vartheta_{n}=\vartheta(L_{n})$. Against \eqref{eq:short} the bound $j$
is the stronger one, the difference being
$\vartheta+\cot\vartheta-\pi/2\ge0$.

\begin{theorem}\label{thm:main}
Let $n$ be a positive integer. Every curve $\mathcal{C}\subset\Sph$ of length
$L\le L_{n}$ satisfies
\begin{equation}\label{eq:main}
  \mathcal{J}(\mathcal{C})\ \ge\ j(L_{n})=\frac{\pi}{2n}-\tan\frac{\pi}{4n},
\end{equation}
and every curve attaining equality has length exactly $L_{n}$. Equality holds
for the sphere-filling ropes $\bta{n}{k}$ of Section~\ref{sec:curves}, where
$0\le k<n$ and $\gcd(k,n)=1$. They have length $L_{n}$, so in particular
$J(L_{n})=j(L_{n})$, and at Kimberling's length $L_{3}=4\pi$ the minimum
is $\pi/6+\sqrt{3}-2$, attained by the six-arc curve $\bta{3}{1}$ of
\eqref{eq:six-arc} and by its mirror image $\bta{3}{2}$.
\end{theorem}

This answers Kimberling's case $L=4\pi$; for other $L>2\pi$ see
Proposition~\ref{prop:other}.

The lower bound rests on the tube estimate
\begin{equation}\label{eq:tube-intro}
  \area(T_{\theta}(\mathcal{C}))\ \le\ 2L\sin\theta
  \qquad(L\ge2\pi,\quad 0\le\theta\le\vartheta),
\end{equation}
where $T_{\theta}(\mathcal{C})$ is the set of points within $\theta$ of
$\mathcal{C}$; at $\theta=\vartheta$ the bound is $4\pi$. For curves of thickness
at least $\sin\theta$ in the sense of Gonzalez and Maddocks
\cite{GM99} (recalled in Section~\ref{sec:curves}), the $\theta$-tube has area
exactly $2L\sin\theta$ \cite[Prop.~1]{GvdM2011b}: that is the area of a
$\theta$-tube which nowhere overlaps itself, and \eqref{eq:tube-intro} says
overlapping can only cost area. A general rectifiable curve has no thickness
and no normals to unroll along, so instead we use the two collars that
$T_{\theta}$ cuts from the complementary disks of $\mathcal{C}$. On a disk $\Omega$ with $C^{2}$ boundary, the inner-parallel
inequality \cite{Hartman64} makes the collar area $f$ satisfy
$f''+f\le\area(\Omega)-2\pi$, and a Wronskian comparison yields the collar
bound (Lemma~\ref{lem:collar}). The disk of area at least $2\pi$ obeys the bound out
to depth $\pi/2$, and the two bounds combine because the complementary areas sum
to $4\pi$.

Section~\ref{sec:bound} develops the collar estimate and draws from it three
conclusions: a covering bound (Corollary~\ref{cor:cover}), the tube estimate
\eqref{eq:tube-intro}, and the inequality $\mathcal{J}\ge j(L)$
(Corollary~\ref{cor:bound}). Section~\ref{sec:equality} analyzes equality: a curve
attaining $j(L)$ is sphere-filling and bisects $\Sph$ into two disks of area
$2\pi$ and inradius $\vartheta$ (Proposition~\ref{prop:equality}), a minimizer of
length at most $L_{n}$ has length exactly $L_{n}$
(Corollary~\ref{cor:exact}), and for $n=1$ it is a great circle
(Corollary~\ref{cor:greatcircle}). Section~\ref{sec:curves} describes the
curves $\bta{n}{k}$, completes the proof of Theorem~\ref{thm:main}, writes down
the six-arc minimizer at $L=4\pi$ and classifies the minimizers of thickness at
least $\sin\vartheta$ (Proposition~\ref{prop:classify}). Section~\ref{sec:other}
brackets $J$ at the remaining lengths.

\section{Short curves}\label{sec:short}

\begin{proof}[Proof of Proposition~\ref{prop:short}]
Parametrize $\mathcal{C}$ by arclength as $\gamma\colon\R/L\mathbb{Z}\to\Sph$, fix
$x\in\Sph$ and put $\psi_{x}(s)=\dist(x,\gamma(s))$, a $1$-Lipschitz function.
Its minimum is $\dist(x,\mathcal{C})$ and its maximum is
$\pi-\dist(-x,\mathcal{C})$, because $\dist(-x,\,\cdot\,)=\pi-\dist(x,\,\cdot\,)$
turns the maximum of the one into the minimum of the other. A continuous
function on a circle has total variation at least twice its oscillation, so
\begin{equation}\label{eq:tv}
  \int_{0}^{L}\lvert\psi_{x}'(s)\rvert\,ds\ \ge\
  2\bigl(\pi-\dist(x,\mathcal{C})-\dist(-x,\mathcal{C})\bigr).
\end{equation}
At every $s$ where $\gamma'(s)$ exists and $x\ne\pm\gamma(s)$ the chain rule
applies to $\psi_{x}(s)=\arccos\langle x,\gamma(s)\rangle$ and gives
$\lvert\psi_{x}'(s)\rvert=\lvert\cos\varphi\rvert$, where $\varphi$ is the azimuth
of $x$ in geodesic polar coordinates centered at $\gamma(s)$, measured from
$\gamma'(s)$. Both exceptional sets are null, the second because $\gamma'$ would
vanish at a density point of $\gamma^{-1}(\{\pm x\})$. Integrating,
\[
  \int_{\Sph}\lvert\cos\varphi\rvert\,dA(x)
  =\int_{0}^{\pi}\!\!\int_{0}^{2\pi}\lvert\cos\varphi\rvert\,d\varphi\,\sin\theta\,d\theta
  =8 ,
\]
so Fubini gives $\int_{\Sph}\int_{0}^{L}\lvert\psi_{x}'(s)\rvert\,ds\,dA(x)=8L$. The
antipodal map preserves $dA$, so each of the two distance integrals equals
$4\pi\mathcal{J}(\mathcal{C})$, and integrating \eqref{eq:tv} over $\Sph$ gives
$8L\ge8\pi^{2}-16\pi\mathcal{J}(\mathcal{C})$, which is \eqref{eq:short}.

Let now $L\le2\pi$ and let $\mathcal{C}$ be the circle of geodesic radius
$\theta_{0}=\arcsin(L/2\pi)\le\pi/2$ about a fixed center. A point at geodesic
distance $\theta$ from the center is at distance $\lvert\theta-\theta_{0}\rvert$
from $\mathcal{C}$, so
\[
  4\pi \mathcal{J}(\mathcal{C})
  =2\pi\int_{0}^{\pi}\lvert\theta-\theta_{0}\rvert\sin\theta\,d\theta
  =2\pi(\pi-2\sin\theta_{0}),
\]
that is $\mathcal{J}(\mathcal{C})=\pi/2-\sin\theta_{0}=\pi/2-L/(2\pi)$.

Finally, let $\mathcal{C}$ be a curve of length $L\le2\pi$ with
$\mathcal{J}(\mathcal{C})=\pi/2-L/(2\pi)$. Every step above except \eqref{eq:tv}
was an identity, so \eqref{eq:tv} is an equality for almost every $x$. For such
an $x$, each of the two arcs cut from the parameter circle by a minimum and a
maximum point of $\psi_{x}$ carries variation exactly equal to the oscillation,
so $\psi_{x}$ is monotone on both. A level strictly between the extremes is
therefore attained on at most one subinterval of each arc, and only countably
many levels are attained on a nondegenerate one. So, writing
$B(x,r)=\{\dist(\,\cdot\,,x)<r\}$ for the open geodesic ball, for almost every
$x$ and almost every $r$ the circle $\partial B(x,r)$ meets $\mathcal{C}$ in at
most two points.

Suppose $\mathcal{C}$ were not contained in a circle. Any three of its points
determine a circle, which some fourth point of $\mathcal{C}$ misses, so
$\mathcal{C}$ has four points that are not concyclic, hence not coplanar.
Label them $y_{1},y_{2},y_{3},y_{4}$ in the cyclic order in which they occur
along $\mathcal{C}$. Two line segments of $\R^{3}$ that meet have coplanar
endpoints, so $[y_{1},y_{3}]$ and $[y_{2},y_{4}]$ are disjoint compact convex
sets and some plane separates them strictly. As the $y_{i}$ lie on $\Sph$, that
plane can be written $\{\langle\,\cdot\,,x\rangle=\cos r\}$ with $x\in\Sph$,
$r\in(0,\pi)$ and $y_{1},y_{3}$ on the side of $x$; then
$y_{1},y_{3}\in B(x,r)$ while $y_{2},y_{4}\notin\overline{B(x,r)}$. Each of the
four arcs of $\mathcal{C}$ between consecutive $y_{i}$ joins a point of
$B(x,r)$ to a point outside $\overline{B(x,r)}$, hence meets
$\partial B(x,r)$; as $\mathcal{C}$ is simple these four arcs meet only in the
$y_{i}$, which lie off $\partial B(x,r)$, so $\partial B(x,r)$ meets
$\mathcal{C}$ in at least four points. The four inclusions are strict, so the
same holds on a neighborhood of $(x,r)$, contradicting the previous paragraph.

So $\mathcal{C}$ lies on a circle, and equals it because a proper closed subset
of a circle is not homeomorphic to one. Its length is $L$, so it is a circle of
geodesic radius $\theta_{0}$.
\end{proof}

\section{The lower bound}\label{sec:bound}

\subsection{Inner parallels}\label{ssec:parallels}

Throughout this subsection $\mathcal{C}\subset\Sph$ is a curve of length $L$
and $\Omega$ is one of the two components of $\Sph\setminus\mathcal{C}$. By
Schoenflies these are disks, and $\partial\Omega=\mathcal{C}$. Put
$\alpha=\area(\Omega)$, let
\begin{equation}\label{eq:f}
  f(\theta)=\area\bigl(\Omega\cap\{\dist(\,\cdot\,,\mathcal{C})\le\theta\}\bigr)
\end{equation}
be the area of the inward collar of depth $\theta$, and let
$r=\max_{\overline{\Omega}}\dist(\,\cdot\,,\mathcal{C})$ be the inradius of
$\Omega$. Thus $f$ is nondecreasing, $f(0)=0$, and $f(\theta)=\alpha$ for
$\theta\ge r$. Since $\Omega$ is open and nonempty, $r>0$, so a point $x$
realizing the maximum lies in $\Omega$, and the open ball $B(x,r)$ misses
$\mathcal{C}$, hence lies in $\Omega$:
\begin{equation}\label{eq:ball}
  B(x,r)\subseteq\Omega,\qquad\text{whence}\qquad 2\pi(1-\cos r)\le\alpha .
\end{equation}
The other component being open and nonempty, $\alpha<4\pi$, so \eqref{eq:ball}
also gives $r<\pi$.

Since $\dist(\,\cdot\,,\mathcal{C})$ is $1$-Lipschitz with unit gradient almost
everywhere on $\Omega$, the coarea formula makes $f$ absolutely continuous, with
$f'(\theta)$ equal almost everywhere to the length of the inward parallel
$\Omega\cap\{\dist(\,\cdot\,,\mathcal{C})=\theta\}$. We state what we need as
Lemma~\ref{lem:parallel} (cf.\ \cite{Hartman64}) and prove it in
Appendix~\ref{app:hartman}.

\begin{lemma}[Hartman]\label{lem:parallel}
Suppose in addition that $\mathcal{C}$ is of class $C^{2}$. Then $f'$ has a
representative of bounded variation on $[0,r]$,
\begin{equation}\label{eq:fprime}
  f'(\theta)=L\cos\theta+(\alpha-2\pi)\sin\theta
  \qquad\text{for all sufficiently small }\theta>0,
\end{equation}
and the distributional derivative of $f'$ satisfies
\begin{equation}\label{eq:ode}
  df'\ \le\ \bigl(\alpha-2\pi-f(\theta)\bigr)\,d\theta
  \qquad\text{on }(0,r).
\end{equation}
\end{lemma}

The parallel length $f'$ can drop abruptly, at depths where the collar reaches
the cut locus of $\mathcal{C}$ in $\Omega$ and part of the parallel disappears;
it never jumps upward, which is what makes the one-sided \eqref{eq:ode} enough.

Comparison with the solution of the corresponding equation turns \eqref{eq:ode}
into the bound we shall use. Write
\begin{equation}\label{eq:F}
  F(\theta)=L\sin\theta+(\alpha-2\pi)(1-\cos\theta),
\end{equation}
the solution of $F''+F=\alpha-2\pi$ with $F(0)=0$ and $F'(0)=L$.

\begin{lemma}[collar bound]\label{lem:collar}
Let $\Omega$, $\alpha$, $r$, $f$, $F$ be as above. Then:
\begin{enumerate}
\item[(a)] $f\le F$ on $[0,r]$, and on $[0,\pi/2]$ as well if $\alpha\ge2\pi$.
\item[(b)] Let $\alpha\ge2\pi$. If $r\le\pi/2$ then $L\sin r\ge2\pi$, and if
  $L\ge2\pi$ then $r\ge\vartheta$.
\end{enumerate}
\end{lemma}

\begin{proof}
(a) Assume first that $\mathcal{C}$ is of class $C^{2}$. The general case is
deduced by approximation in Appendix~\ref{app:approx}. Put $w=f-F$. It is absolutely continuous,
$w(0)=0$, and by Lemma~\ref{lem:parallel} its derivative has a representative of
bounded variation, which we take right-continuous and denote by $w'$. Since
$F''=\alpha-2\pi-F$, inequality \eqref{eq:ode} gives
$dw'\le(\alpha-2\pi-f-F'')\,d\theta=(F-f)\,d\theta=-w\,d\theta$, that is
\begin{equation}\label{eq:dwprime}
  dw'+w\,d\theta\ \le\ 0\qquad\text{on }(0,r),
\end{equation}
which makes $w$ a subsolution of $y''+y=0$. It also starts at rest: by
\eqref{eq:fprime}, $w'$ vanishes on some interval $(0,\theta_{0})$ with
$\theta_{0}\in(0,r)$, and since $w(0)=0$ this gives $w\equiv0$ on
$[0,\theta_{0}]$.

Consider the Wronskian $W=w'\sin\theta-w\cos\theta$. Since
$d(w'\sin\theta)=\sin\theta\,dw'+w'\cos\theta\,d\theta$ and
$d(w\cos\theta)=(w'\cos\theta-w\sin\theta)\,d\theta$, the two terms in
$w'\cos\theta\,d\theta$ cancel and, as measures on $(0,r)$,
$dW=\sin\theta\,(dw'+w\,d\theta)\le0$ by \eqref{eq:dwprime} and
$\sin\theta>0$ on $(0,r)\subseteq(0,\pi)$. So $W$ is
nonincreasing on $(0,r)$. It vanishes on $(0,\theta_{0})$, whence $W\le0$
throughout. Now $w/\sin\theta$ is absolutely continuous on compact subintervals
of $(0,r)$, with derivative $W/\sin^{2}\theta\le0$ almost everywhere, and it
vanishes on $(0,\theta_{0})$, so it is nonpositive on $(0,r)$. Therefore
$w\le0$ on $(0,r)$, and on $[0,r]$ by continuity.

For the second half of (a), let $\alpha\ge2\pi$. Then
$F'(\theta)=L\cos\theta+(\alpha-2\pi)\sin\theta\ge0$ on $[0,\pi/2]$, so $F$ is
nondecreasing there. For $\theta\le\min\{r,\pi/2\}$ the first half applies
directly, while for $r\le\theta\le\pi/2$ it gives
$f(\theta)=\alpha=f(r)\le F(r)\le F(\theta)$.

(b) Taking $\theta=r$ in (a), where $f(r)=\alpha$, gives
$\alpha\le L\sin r+(\alpha-2\pi)(1-\cos r)$, that is
$2\pi+(\alpha-2\pi)\cos r\le L\sin r$. Let $\alpha\ge2\pi$. If $r\le\pi/2$ then
$\cos r\ge0$, so the left-hand side is at least $2\pi$ and $L\sin r\ge2\pi$. Let
now also $L\ge2\pi$. If $r>\pi/2$ then $r>\pi/2\ge\vartheta$; and if
$r\le\pi/2$ then $L\sin r\ge2\pi=L\sin\vartheta$ by \eqref{eq:vartheta}, so
$\sin r\ge\sin\vartheta$ and, $r$ and $\vartheta$ both lying in $[0,\pi/2]$
where $\sin$ increases, $r\ge\vartheta$.
\end{proof}

\subsection{The tube bound}\label{ssec:tube}

For a curve $\mathcal{C}\subset\Sph$, put
$T_{\theta}(\mathcal{C})=\{\dist(\,\cdot\,,\mathcal{C})\le\theta\}$ and
$A_{\mathcal{C}}(\theta)=\area(T_{\theta}(\mathcal{C}))$. Since
$\dist(\,\cdot\,,\mathcal{C})$ is continuous,
nonnegative and bounded by $\pi$, the layer-cake formula turns \eqref{eq:Jcal} into
\begin{equation}\label{eq:layercake}
  \mathcal{J}(\mathcal{C})=\frac{1}{4\pi}\int_{0}^{\pi}\bigl(4\pi-A_{\mathcal{C}}(\theta)\bigr)\,d\theta
      =\int_{0}^{\pi}\Bigl(1-\frac{A_{\mathcal{C}}(\theta)}{4\pi}\Bigr)d\theta .
\end{equation}
Minimizing $\mathcal{J}$ is therefore the same as making $A_{\mathcal{C}}$ as
large as possible at every radius; \eqref{eq:tube} below is the cap.

Let
$\Omega_{1},\Omega_{2}$ be the two components of $\Sph\setminus\mathcal{C}$,
\emph{labeled so that} $\alpha_{1}\le\alpha_{2}$, where
$\alpha_{i}=\area(\Omega_{i})$. Thus
\begin{equation}\label{eq:areas}
  \alpha_{1}+\alpha_{2}=4\pi,\qquad \alpha_{1}\le2\pi\le\alpha_{2} .
\end{equation}
Write $r_{i}$ for the inradius of $\Omega_{i}$, $f_{i}$ for its collar area
\eqref{eq:f} and $F_{i}$ for the function \eqref{eq:F} attached to it. By
\eqref{eq:areas} the two comparison functions add up to
\begin{equation}\label{eq:Fsum}
  F_{1}+F_{2}=2L\sin\theta .
\end{equation}
Since
$\Sph=\overline{\Omega_{1}}\cup\overline{\Omega_{2}}$ and a rectifiable curve is
a null set,
\begin{equation}\label{eq:split}
  A_{\mathcal{C}}=f_{1}+f_{2},
  \qquad
  R:=\max_{\Sph}\dist(\,\cdot\,,\mathcal{C})=\max\{r_{1},r_{2}\} .
\end{equation}
We call $R$ the \emph{covering radius} of $\mathcal{C}$. By \eqref{eq:split} and
the absolute continuity of the collars, $A_{\mathcal{C}}$ is continuous, so the
level sets of $\dist(\,\cdot\,,\mathcal{C})$ are null and open and closed tubes
have the same area; the results of \cite{GvdM2011b} used below are stated for
the open ones.

\begin{corollary}\label{cor:cover}
Every curve $\mathcal{C}\subset\Sph$ of length $L\ge2\pi$ has covering radius
$R\ge\vartheta$. Moreover, whatever the length, if $T_{\theta}(\mathcal{C})=\Sph$
for some $\theta\in(0,\pi/2]$ then $L\ge2\pi/\sin\theta$.
\end{corollary}

\begin{proof}
By \eqref{eq:areas} the area of $\Omega_{2}$ is at least $2\pi$, so
Lemma~\ref{lem:collar}(b) gives $r_{2}\ge\vartheta$ when $L\ge2\pi$, and hence
$R\ge\vartheta$ by \eqref{eq:split}. Now $T_{\theta}(\mathcal{C})=\Sph$ says
exactly that $R\le\theta$, so $r_{2}\le\theta\le\pi/2$ and the same part of that
lemma gives $L\sin r_{2}\ge2\pi$. As $\sin$ increases on $[0,\pi/2]$,
$L\sin\theta\ge L\sin r_{2}\ge2\pi$.
\end{proof}

By Lemma~\ref{lem:tube-n} below the bound $L\ge2\pi/\sin\theta$ is attained at
every $\theta=\vartheta_{n}$.

\begin{proposition}\label{prop:tube}
Every curve $\mathcal{C}\subset\Sph$ of length $L\ge2\pi$ satisfies
\begin{equation}\label{eq:tube}
  A_{\mathcal{C}}(\theta)\ \le\ 2L\sin\theta,
  \qquad 0\le\theta\le\vartheta .
\end{equation}
\end{proposition}

\begin{proof}
By \eqref{eq:split} and \eqref{eq:Fsum} it is enough to prove that
$f_{i}\le F_{i}$ on all of $[0,\vartheta]$ for $i=1,2$, for then
$A_{\mathcal{C}}\le F_{1}+F_{2}=2L\sin\theta$ there.

For $i=2$ this is Lemma~\ref{lem:collar}(a), since $\alpha_{2}\ge2\pi$ and
$\vartheta\le\pi/2$. For $i=1$ the same part covers $[0,r_{1}]$, so let
$r_{1}\le\theta\le\vartheta$, where $f_{1}(\theta)=\alpha_{1}$; we must see that
$F_{1}\ge\alpha_{1}$ there. As $2\pi-\alpha_{1}\ge0$ by \eqref{eq:areas},
$F_{1}''(\theta)=-L\sin\theta-(2\pi-\alpha_{1})\cos\theta\le0$ on $[0,\pi/2]$,
so $F_{1}$ is concave on $[r_{1},\vartheta]$ and it is enough to check the two
endpoints. At $\theta=r_{1}$, Lemma~\ref{lem:collar}(a) gives
$F_{1}(r_{1})\ge f_{1}(r_{1})=\alpha_{1}$; and at $\theta=\vartheta$, where
$L\sin\vartheta=2\pi$,
$F_{1}(\vartheta)=2\pi-(2\pi-\alpha_{1})(1-\cos\vartheta)
=\alpha_{1}+(2\pi-\alpha_{1})\cos\vartheta\ge\alpha_{1}$.
\end{proof}

\begin{corollary}\label{cor:bound}
Every curve $\mathcal{C}\subset\Sph$ of length $L\ge2\pi$ satisfies
$\mathcal{J}(\mathcal{C})\ge j(L)$.
\end{corollary}

\begin{proof}
Discard the part of \eqref{eq:layercake} beyond $\vartheta$, which is
nonnegative because $A_{\mathcal{C}}\le4\pi$, and insert \eqref{eq:tube},
recalling that $2L=4\pi/\sin\vartheta$:
\[
  \mathcal{J}(\mathcal{C})\ \ge\
  \int_{0}^{\vartheta}\Bigl(1-\frac{\sin\theta}{\sin\vartheta}\Bigr)d\theta
  \ =\ \vartheta-\frac{1-\cos\vartheta}{\sin\vartheta}
  \ =\ \vartheta-\tan\frac{\vartheta}{2}\ =\ j(L).
  \qedhere
\]
\end{proof}

\section{Curves that attain the bound}\label{sec:equality}

Call a curve $\mathcal{C}\subset\Sph$ \emph{sphere-filling} if its tube has, at
every radius $\theta\in[0,\vartheta]$, the largest area that \eqref{eq:tube}
allows:
\begin{equation}\label{eq:spherefilling}
  A_{\mathcal{C}}(\theta)=2L\sin\theta,
  \qquad 0\le\theta\le\vartheta .
\end{equation}
At $\theta=\vartheta$ this reads $A_{\mathcal{C}}(\vartheta)=4\pi$. As
$T_{\vartheta}(\mathcal{C})$
is closed, its complement is then an open null set, hence empty, and
$T_{\vartheta}(\mathcal{C})=\Sph$. The name is borrowed from Gerlach and von der
Mosel; here no thickness is assumed. Lemma~\ref{lem:tube-n} below shows that
their ropes satisfy \eqref{eq:spherefilling}.

\begin{proposition}\label{prop:equality}
Let $\mathcal{C}\subset\Sph$ be a curve of length $L\ge2\pi$ with
$\mathcal{J}(\mathcal{C})=j(L)$. Then $\mathcal{C}$ is
sphere-filling, its covering radius is $\vartheta$, and the two components
$\Omega_{1},\Omega_{2}$ of $\Sph\setminus\mathcal{C}$ satisfy
\begin{equation}
  \area(\Omega_{i})=2\pi,\qquad r_{i}=\vartheta,\qquad
  f_{i}(\theta)=L\sin\theta\ \ \text{for }0\le\theta\le\vartheta .
\end{equation}
\end{proposition}

\begin{proof}
\emph{Step 1: $\mathcal{C}$ is sphere-filling and $R=\vartheta$.} By
\eqref{eq:layercake} and the proof of Corollary~\ref{cor:bound},
\[
  \mathcal{J}(\mathcal{C})-j(L)
  =\frac{1}{4\pi}\int_{0}^{\vartheta}
     \bigl(2L\sin\theta-A_{\mathcal{C}}(\theta)\bigr)d\theta
   +\int_{\vartheta}^{\pi}\Bigl(1-\frac{A_{\mathcal{C}}(\theta)}{4\pi}\Bigr)d\theta ,
\]
and both integrands are nonnegative, by Proposition~\ref{prop:tube} and by
$A_{\mathcal{C}}\le4\pi$. Equality therefore forces
$A_{\mathcal{C}}(\theta)=2L\sin\theta$ for almost every $\theta\in(0,\vartheta)$,
and hence for every $\theta\in[0,\vartheta]$, both sides being continuous
(Subsection~\ref{ssec:tube}). At $\theta=\vartheta$ this reads
$A_{\mathcal{C}}(\vartheta)=2L\sin\vartheta=4\pi$. So \eqref{eq:spherefilling} holds and
$T_{\vartheta}(\mathcal{C})=\Sph$, that is $R\le\vartheta$, while
$R\ge\vartheta$ by Corollary~\ref{cor:cover}; so $R=\vartheta$.

\emph{Step 2: $\alpha_{1}=\alpha_{2}=2\pi$.} Keep the labeling
$\alpha_{1}\le2\pi\le\alpha_{2}$ of Subsection~\ref{ssec:tube}. Since
$\alpha_{2}\ge2\pi$, Lemma~\ref{lem:collar}(a) bounds $f_{2}$ by $F_{2}$ on
$[0,\vartheta]$, while $f_{1}\le\alpha_{1}$ always, so
\eqref{eq:spherefilling}, \eqref{eq:split} and \eqref{eq:Fsum} give
\[
  F_{1}(\theta)+F_{2}(\theta)=2L\sin\theta=A_{\mathcal{C}}(\theta)
  =f_{1}(\theta)+f_{2}(\theta)\le\alpha_{1}+F_{2}(\theta),
  \qquad 0\le\theta\le\vartheta ,
\]
that is $F_{1}\le\alpha_{1}$ on $[0,\vartheta]$, which since $L\sin\vartheta=2\pi$
reads $(2\pi-\alpha_{1})\cos\theta\le L(\sin\vartheta-\sin\theta)$. Since $\cos$
is decreasing on $[0,\pi/2]$ we have
$\sin\vartheta-\sin\theta\le(\vartheta-\theta)\cos\theta$, so, dividing by
$\cos\theta>0$,
$2\pi-\alpha_{1}\le L(\vartheta-\theta)$ for every
$\theta\in[0,\vartheta)$. Letting $\theta\uparrow\vartheta$ gives
$\alpha_{1}\ge2\pi$, hence $\alpha_{1}=\alpha_{2}=2\pi$.

\emph{Step 3: collars and inradii.} With $\alpha_{i}=2\pi$ we have
$F_{i}(\theta)=L\sin\theta$, so Lemma~\ref{lem:collar}(a) gives
$f_{i}(\theta)\le L\sin\theta$ on
$[0,\pi/2]$, while $f_{1}(\theta)+f_{2}(\theta)=A_{\mathcal{C}}(\theta)=2L\sin\theta$
for $\theta\le\vartheta$, hence $f_{i}(\theta)=L\sin\theta$ there. At
$\theta=\vartheta$ this reads
$f_{i}(\vartheta)=L\sin\vartheta=2\pi=\alpha_{i}$, so
$\Omega_{i}\cap\{\dist(\,\cdot\,,\mathcal{C})>\vartheta\}$ is an open null set, hence
empty, and $r_{i}\le\vartheta$. Finally $r_{i}\ge\vartheta$ by
Lemma~\ref{lem:collar}(b), since $\alpha_{i}=2\pi$.
\end{proof}

\begin{corollary}\label{cor:exact}
Let $n$ be a positive integer. Every curve $\mathcal{C}\subset\Sph$ of length
$L\le L_{n}$ satisfies $\mathcal{J}(\mathcal{C})\ge j(L_{n})$, and equality forces
$L=L_{n}$.
\end{corollary}

\begin{proof}
If $L<2\pi$ then \eqref{eq:short} gives
$\mathcal{J}(\mathcal{C})>\pi/2-1=j(2\pi)\ge j(L_{n})$; and if $2\pi\le L\le L_{n}$
then Corollary~\ref{cor:bound} and the monotonicity of $j$ give
$\mathcal{J}(\mathcal{C})\ge j(L)\ge j(L_{n})$. Either way
$\mathcal{J}(\mathcal{C})\ge j(L_{n})$, and equality excludes the first case, so
$2\pi\le L\le L_{n}$ and $j(L)=j(L_{n})$. As $j$ is strictly decreasing on
$[2\pi,\infty)$, this forces $L=L_{n}$.
\end{proof}

\begin{corollary}\label{cor:greatcircle}
Every curve $\mathcal{C}\subset\Sph$ of length $L\le2\pi$ satisfies
$\mathcal{J}(\mathcal{C})\ge\pi/2-1$, with equality if and only if $\mathcal{C}$ is a
great circle. In particular Theorem~\ref{thm:main} holds for $n=1$, with the
great circle as its only minimizer.
\end{corollary}

\begin{proof}
By \eqref{eq:short}, $\mathcal{J}(\mathcal{C})\ge\pi/2-L/(2\pi)\ge\pi/2-1$, and
the second inequality is strict unless $L=2\pi$. So equality forces $L=2\pi$ and
$\mathcal{J}(\mathcal{C})=\pi/2-L/(2\pi)$, which by Proposition~\ref{prop:short}
holds for the circle of length $2\pi$, a great circle, and for no other curve.
\end{proof}

\section{The sphere-filling ropes}\label{sec:curves}

The \emph{thickness} of a closed curve in $\R^{3}$ is, in the sense of Gonzalez
and Maddocks \cite{GM99} (see also \cite[(1.1)]{GvdM2011b}), the infimum over
triples of distinct points on it of the radius of the smallest circle through
those three points. We recall the construction of Gerlach and von der Mosel
\cite{GvdM2011b}.

Fix a positive integer $n$ and a $k\in\{0,\dots,n-1\}$ with $\gcd(k,n)=1$.
Start with the $n$ latitude circles of $\Sph$, the $i$th at colatitude
$(2i+1)\vartheta_{n}$ for $0\le i\le n-1$: consecutive ones are at distance
$2\vartheta_{n}$, and the extreme ones are the circles of geodesic radius
$\vartheta_{n}$ about the poles. Cut $\Sph$ along a great
circle through the poles into two closed hemispheres, which meet each latitude
circle orthogonally and cut it into two semicircles. Keep one hemisphere fixed
and rotate the other, about the axis of the cutting circle, through the angle
$2k\vartheta_{n}$. Every endpoint of a fixed semicircle then meets exactly one
endpoint of a rotated one, and the $2n$ semicircles join up into
$\gcd(k,n)$ closed loops \cite[Lemmas~1--2]{GvdM2011b}. The hypothesis
$\gcd(k,n)=1$ leaves a single loop $\bta{n}{k}$. The value $k=0$ occurs only
for $n=1$, where the construction is vacuous and $\bta{1}{0}$ is a great
circle.

For $n\ge2$, \cite[Lemma~1]{GvdM2011b} shows that $\bta{n}{k}$ is a simple
closed piecewise circular curve whose constant-speed parametrization is of
class $C^{1,1}$ and whose thickness is $\sin\vartheta_{n}$. For $n=1$ this is
the great circle, of thickness $1=\sin\vartheta_{1}$. Up to congruence, the
$\bta{n}{k}$ are exactly the longest closed curves on $\Sph$ of thickness at
least $\sin\vartheta_{n}$ \cite[Thms.~2--3]{GvdM2011b}.

Figure~\ref{fig:minimizer} shows $\bta{3}{1}$ together with the tubes computed
in the next lemma.

\begin{figure}[ht]
  \centering
  \includegraphics[width=0.8\textwidth]{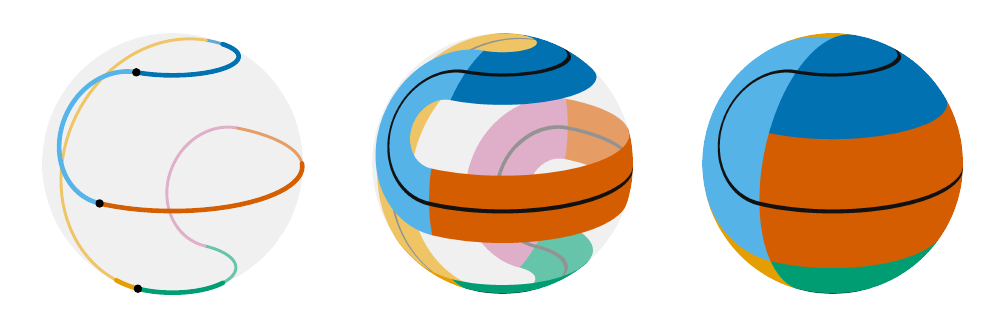}
  \caption{The curve $\bta{3}{1}$, of length $4\pi$ (left), and the six patches
  swept by the geodesic segments of length $\theta$ issuing normally on both
  sides of its six semicircular arcs. These patches have disjoint interiors
  \cite[Lemma~3]{GvdM2011b} and their union is the $\theta$-tube, of area
  $8\pi\sin\theta$: at $\theta<\vartheta_{3}$ (center) it leaves gaps, and at
  $\theta=\vartheta_{3}$ (right) the patches tile $\Sph$.}
  \label{fig:minimizer}
\end{figure}

\begin{lemma}\label{lem:tube-n}
For $n$ and $k$ as above, $\length(\bta{n}{k})=L_{n}$ and
$A_{\bta{n}{k}}(\theta)=2L_{n}\sin\theta$ for $0\le\theta\le\vartheta_{n}$.
Consequently $T_{\vartheta_{n}}(\bta{n}{k})=\Sph$ and
$\mathcal{J}(\bta{n}{k})=j(L_{n})$.
\end{lemma}

\begin{proof}
By \cite[Prop.~1]{GvdM2011b}, if a closed rectifiable curve has thickness at
least $\sin\theta$ for some $\theta\in(0,\pi/2]$, its $\theta$-tube has area
$2\sin\theta$ times the length of the curve, and by \cite[(2.4)]{GvdM2011b} the
$\vartheta_{n}$-tube of $\bta{n}{k}$ has area $4\pi$. As the thickness of
$\bta{n}{k}$ is $\sin\vartheta_{n}$, the two give
$2\length(\bta{n}{k})\sin\vartheta_{n}=4\pi$, that is
$\length(\bta{n}{k})=L_{n}$, and hence $A_{\bta{n}{k}}(\theta)=2L_{n}\sin\theta$
throughout $[0,\vartheta_{n}]$.

Thus $\bta{n}{k}$ is sphere-filling, and $T_{\vartheta_{n}}(\bta{n}{k})=\Sph$ as
observed after \eqref{eq:spherefilling}. Therefore $A_{\bta{n}{k}}\equiv4\pi$ beyond
$\vartheta_{n}$, and \eqref{eq:layercake} gives
$\mathcal{J}(\bta{n}{k})=\int_{0}^{\vartheta_{n}}
 (1-\sin\theta/\sin\vartheta_{n})\,d\theta
=\vartheta_{n}-\tan\frac{\vartheta_{n}}{2}=j(L_{n})$.
\end{proof}

\begin{proof}[Proof of Theorem~\ref{thm:main}]
Inequality \eqref{eq:main} and the assertion that equality forces $L=L_{n}$ are
Corollary~\ref{cor:exact}. Equality does hold for
$\bta{n}{k}$, which by Lemma~\ref{lem:tube-n} has length $L_{n}$ and
$\mathcal{J}(\bta{n}{k})=j(L_{n})$. The case $n=3$, $k=1$ is
the curve \eqref{eq:six-arc} below, and $\bta{3}{2}$ is its mirror image by
Remark~\ref{rem:mirror}.
\end{proof}

The first values of $L_{n}=2\pi/\sin\frac{\pi}{2n}$ and
$J(L_{n})=\frac{\pi}{2n}-\tan\frac{\pi}{4n}$ are
{\tiny
\begin{equation}
\begin{aligned}
  L_{1}&=2\pi, & J(L_{1})&=\tfrac{\pi}{2}-1&&=0.570796\ldots, \\
  L_{2}&=2\pi\sqrt{2}, & J(L_{2})&=\tfrac{\pi}{4}-\sqrt{2}+1&&=0.371184\ldots, \\
  L_{3}&=4\pi, & J(L_{3})&=\tfrac{\pi}{6}+\sqrt{3}-2&&=0.255649\ldots, \\
  L_{4}&=2\pi\sqrt{4+2\sqrt{2}}, &
       J(L_{4})&=\tfrac{\pi}{8}+1+\sqrt{2}-\sqrt{4+2\sqrt{2}}&&=0.193786\ldots.
\end{aligned}
\end{equation}
}
\subsection{Kimberling's length}

In this subsection $x,y,z$ are ambient Cartesian coordinates on $\R^{3}$.
Take $n=3$, so that $\vartheta_{3}=\pi/6$, the three
latitude circles are at colatitudes $\pi/6$, $\pi/2$, $5\pi/6$ and $L_{3}=4\pi$,
and take $k=1$. Let the cutting circle be $\{y=0\}$, so that the turning angle
$2k\vartheta_{3}=\pi/3$ is measured about the $y$-axis and the turning rotation is
\begin{equation}
  Q(x,y,z)=\Bigl(\tfrac{1}{2}x+\tfrac{\sqrt3}{2}z,\ y,\
  -\tfrac{\sqrt3}{2}x+\tfrac{1}{2}z\Bigr).
\end{equation}
Parametrize the three latitude circles by
\begin{equation}
  c_{0}(t)=\bigl(\tfrac12\cos t,\ \tfrac12\sin t,\ \tfrac{\sqrt3}{2}\bigr),
  \quad
  c_{1}(t)=(\cos t,\ \sin t,\ 0),
  \quad
  c_{2}(t)=\bigl(\tfrac12\cos t,\ \tfrac12\sin t,\ -\tfrac{\sqrt3}{2}\bigr),
\end{equation}
so that $c_{i}([0,\pi])$ is the semicircle in $\{y\ge0\}$, which stays fixed,
and $c_{i}([-\pi,0])$ the one in $\{y\le0\}$, which is turned by $Q$.
Traversing the six semicircles of $\bta{3}{1}$ in the order
$c_{0}$, $Qc_{1}$, $c_{2}$, $Qc_{2}$, $c_{1}$, $Qc_{0}$ yields the unit-speed
parametrization $\gamma\colon[0,4\pi]\to\Sph$ asked for in Problem~10,
{\tiny
\begin{equation}\label{eq:six-arc}
  \gamma(s)=
  \begin{cases}
    c_{0}(2s), & 0\le s\le\tfrac{\pi}{2},\\[2pt]
    Q\,c_{1}\bigl(s-\tfrac{3\pi}{2}\bigr), & \tfrac{\pi}{2}\le s\le\tfrac{3\pi}{2},\\[2pt]
    c_{2}(2s-3\pi), & \tfrac{3\pi}{2}\le s\le2\pi,\\[2pt]
    Q\,c_{2}(4\pi-2s), & 2\pi\le s\le\tfrac{5\pi}{2},\\[2pt]
    c_{1}\bigl(\tfrac{7\pi}{2}-s\bigr), & \tfrac{5\pi}{2}\le s\le\tfrac{7\pi}{2},\\[2pt]
    Q\,c_{0}(7\pi-2s), & \tfrac{7\pi}{2}\le s\le4\pi .
  \end{cases}
\end{equation}
}
\begin{remark}[mirror images]\label{rem:mirror}
For an integer $n\ge2$ and $1\le k<n$ with $\gcd(k,n)=1$, the curves $\bta{n}{k}$
and $\bta{n}{n-k}$ are mirror images. Indeed, reflection in the cutting plane
fixes each latitude circle, exchanges the two hemispheres and reverses the
turning angle, so it carries $\bta{n}{k}$ to a curve congruent to the one built
with turning angle $-2k\vartheta_{n}$. Turning the moving hemisphere by a
further $\pi$ about the same axis merely permutes its semicircles, exchanging
the $i$th and $(n-1-i)$th latitude circles, so turning angles differing by
$\pi$ produce the same curve, and $-2k\vartheta_{n}+\pi=2(n-k)\vartheta_{n}$.

By \cite[Lemma~1]{GvdM2011b} no rigid motion carries $\bta{n}{k_{1}}$ to
$\bta{n}{k_{2}}$ when $k_{1}\ne k_{2}$, so $\bta{n}{k}$ and $\bta{n}{n-k}$ are
mirror images that are not rotations of one another unless $k=n-k$, which under
$\gcd(k,n)=1$ occurs only for $n=2$. In particular $\bta{3}{1}$ and
$\bta{3}{2}$ are mirror images but are not congruent by a rotation.
\end{remark}

\begin{proposition}\label{prop:classify}
Let $\mathcal{C}\subset\Sph$ be a curve of length $L\ge2\pi$ with
$\mathcal{J}(\mathcal{C})=j(L)$ and thickness at least $\sin\vartheta$. Then
$\vartheta=\vartheta_{n}$ and $L=L_{n}$ for some positive integer $n$, and
$\mathcal{C}$ is congruent to one of the curves $\bta{n}{k}$.
\end{proposition}

\begin{proof}
By Proposition~\ref{prop:equality} the $\vartheta$-tube of $\mathcal{C}$ has
area $2L\sin\vartheta=4\pi$, so $\mathcal{C}$ is sphere-filling in the sense of
Gerlach and von der Mosel. A curve of positive thickness has a $C^{1,1}$ arclength parametrization with
nonvanishing derivative \cite[\S1]{GvdM2011b}, so $\mathcal{C}$ satisfies the
hypotheses of \cite[Thm.~4]{GvdM2011b} with $\vartheta\in(0,\pi/2]$ and
thickness at least $\sin\vartheta$. That theorem gives $\vartheta=\vartheta_{n}$ and
$\mathcal{C}=\bta{n}{k}$
up to congruence, for some positive integer $n$ and some admissible $k$. Finally
$L=2\pi/\sin\vartheta=L_{n}$.
\end{proof}

\section{Other lengths}\label{sec:other}

\begin{proposition}\label{prop:other}
The function $J$ is nonincreasing on $(0,\infty)$. Consequently
\begin{equation}\label{eq:bracket}
  j(L)\ \le\ J(L)\ \le\ j(L_{n})\qquad
  \text{for every positive integer }n\text{ with }L\ge L_{n}.
\end{equation}
\end{proposition}

\begin{proof}
Let $d_{H}$ denote the Hausdorff distance. Since
$\lvert\dist(x,\mathcal{C})-\dist(x,\mathcal{C}')\rvert
\le d_{H}(\mathcal{C},\mathcal{C}')$ for every $x\in\Sph$, the functional $\mathcal{J}$ is
$1$-Lipschitz for $d_{H}$.

Let $0<L<L'$, let $\mathcal{C}$ be a curve of length $L$, and let
$\delta>0$. Apply Proposition~\ref{prop:approx} with an index $j$ so large that
$\delta_{j}<\min\{\delta,\,L'-L\}$: it produces a $C^{\infty}$ simple closed
curve $\mathcal{C}_{0}$ with $d_{H}(\mathcal{C}_{0},\mathcal{C})<\delta$ and
$\length(\mathcal{C}_{0})\le L+\delta_{j}<L'$. Lemma~\ref{lem:wiggle} then provides a
simple closed curve $\mathcal{C}'$ of length exactly $L'$ with
$d_{H}(\mathcal{C}',\mathcal{C}_{0})<\delta$. As $\mathcal{J}$ is $1$-Lipschitz for
$d_{H}$,
\[
  J(L')\ \le\ \mathcal{J}(\mathcal{C}')\ \le\ \mathcal{J}(\mathcal{C})+2\delta .
\]
Taking the infimum over $\mathcal{C}$ gives
$J(L')\le J(L)+2\delta$, and $\delta\downarrow0$ proves the
monotonicity. In \eqref{eq:bracket} the left inequality is
Corollary~\ref{cor:bound} and the right one follows from monotonicity and
Theorem~\ref{thm:main}, which give $J(L)\le J(L_{n})=j(L_{n})$.
\end{proof}

\begin{question}
Classify all minimizers at $L=L_{n}$, and determine $J(L)$ for $L>2\pi$ with
$L\notin\{L_{n}\}$. Whether the bound $j(L)$ is attained at such an $L$ is open:
by Proposition~\ref{prop:equality} a curve of that length with
$\mathcal{J}=j(L)$ would bisect $\Sph$ into two disks of inradius $\vartheta(L)$
and would attain the covering bound of Corollary~\ref{cor:cover} at the radius
$\vartheta(L)\notin\{\vartheta_{n}\}$, which Proposition~\ref{prop:classify}
rules out for curves of thickness at least $\sin\vartheta$ but not in general.
Even the existence of a minimizer at a fixed $L$ is open, since
$\{\mathcal{C}:\length(\mathcal{C})=L\}$ is not closed for $d_{H}$ and length is
only lower semicontinuous.
\end{question}

\appendix

\section{The inner-parallel inequality}\label{app:hartman}

Throughout, $\Omega$, $\alpha$, $L$, $f$ and $r$ are as in
Subsection~\ref{ssec:parallels} and $\partial\Omega$ is of class $C^{2}$. By
\eqref{eq:f} the inward collar of depth $\theta$ is
$\Omega\cap T_{\theta}(\mathcal{C})$, of area $f(\theta)$.
Parametrize $\gamma=\partial\Omega$ by arclength $s$, let $\mathbf{n}$ be the
inward unit normal and $\kappa=\langle\gamma'',\mathbf{n}\rangle$ the
corresponding geodesic curvature, and put
$\Phi(s,\theta)=\cos\theta\,\gamma(s)+\sin\theta\,\mathbf{n}(s)$, so that
$\theta\ge0$ measures depth into $\Omega$. Since $\mathbf{n}'=-\kappa\gamma'$,
\begin{equation}\label{eq:fermi}
  \partial_{s}\Phi=\mu(s,\theta)\,\gamma'(s),\qquad
  \mu(s,\theta)=\cos\theta-\kappa(s)\sin\theta,\qquad
  \partial_{\theta}^{2}\mu+\mu=0 ,
\end{equation}
so the parallel metric is $d\theta^{2}+\mu^{2}ds^{2}$, with $\mu(s,0)=1$ and
$\partial_{\theta}\mu(s,0)=-\kappa(s)$. Gauss--Bonnet gives
\begin{equation}\label{eq:gb}
  \int_{0}^{L}\kappa\,ds=2\pi-\alpha .
\end{equation}

Hartman works with a complete metric on $\R^{2}$ whose coefficients are of class
$C^{2}$ and whose curvature $K$ is of class $C^{1}$ \cite[\S1]{Hartman64}.
Stereographic projection from a point of $\Sph\setminus\overline{\Omega}$ makes
$\mathcal{C}$ a $C^{2}$ Jordan curve in $\R^{2}$ with $\Omega$ inside it, and
carries the spherical metric to coefficients of that regularity, though not to a
complete metric; interpolating smoothly to the Euclidean metric outside a compact
neighborhood of the image of $\overline{\Omega}$ restores completeness and changes
nothing near $\overline{\Omega}$. Truncating a path from a point of $\Omega$ at
its first boundary hit keeps it in $\overline{\Omega}$, where the metric is
unchanged, so
$\dist(\,\cdot\,,\partial\Omega)$ and $f$ are those of $\Sph$.

His metric coefficient is our $\mu$; his normal parameter $p$ is positive
outside the curve \cite[\S2]{Hartman64}, so $p=-\theta$; and his geodesic
curvature, normalized by $\partial_{p}\mu|_{p=0}$ \cite[(2.4)]{Hartman64}, is
our $\kappa$.

\begin{proof}[Proof of Lemma~\ref{lem:parallel}]
A minimizing geodesic from an interior point meets $\mathcal{C}$ orthogonally, so
the part of $\Omega$ reached by the minimizing inward normal arcs of length at
most $\theta$ is exactly the collar $\Omega\cap T_{\theta}(\mathcal{C})$. Thus $f$ is
Hartman's area function and $f'$ his parallel length, of bounded variation
\cite[Prop.~5.1, (6.8), (6.29), (6.30), Thm.~6.2, Cor.~6.2]{Hartman64}. At
$p=-\theta$ his $p$-locus is the inward parallel, of length $f'(\theta)$, and
the region it bounds with $\mathcal{C}$ is the collar
$\Omega\cap T_{\theta}(\mathcal{C})$, so his inward differential inequality
\cite[(6.4)]{Hartman64}, integrated as in \cite[Cor.~6.1]{Hartman64}, gives for
$0<\theta_{1}<\theta_{2}<r$
\[
  f'(\theta_{2})-f'(\theta_{1})\ \le\
  -\int_{\theta_{1}}^{\theta_{2}}
    \Bigl(\int_{0}^{L}\kappa\,ds
      +\int_{\Omega\cap T_{\theta}(\mathcal{C})}K\,dA\Bigr)d\theta
  \ =\ \int_{\theta_{1}}^{\theta_{2}}\bigl(\alpha-2\pi-f(\theta)\bigr)d\theta ,
\]
the last step by \eqref{eq:gb} and $K\equiv1$. This is \eqref{eq:ode}. (His
inward range runs down to the infimum $\bar{p}$ of his inward extreme values
\cite[\S6]{Hartman64}, and $\bar{p}\le-r$ because the deepest point of $\Omega$
lies on a minimizing inward normal arc \cite[Prop.~5.1]{Hartman64}.)

Since $\mathcal{C}$ is $C^{2}$ and compact, $\Phi$ is a $C^{1}$ map whose
Jacobian equals $\mu(s,0)=1$ along $\theta=0$, hence a diffeomorphism of
$(\R/L\mathbb{Z})\times[0,\varepsilon]$ onto a collar for some $\varepsilon>0$;
for $\theta\le\varepsilon$ every inward normal arc of length $\theta$ still
minimizes, so $f'(\theta)=\int_{0}^{L}\mu(s,\theta)\,ds$ there. Hence $f'''+f'=0$
by \eqref{eq:fermi}, with $f'(0)=L$ and
$f''(0)=-\int_{0}^{L}\kappa=\alpha-2\pi$ by \eqref{eq:gb}, and solving gives
\eqref{eq:fprime}. There $f''=\alpha-2\pi-f$, so the displayed inequality is an
equality before the first cut value.
\end{proof}

\section{Approximation}\label{app:approx}

\begin{lemma}\label{lem:wiggle}
Let $\mathcal{C}\subset\Sph$ be a $C^{\infty}$ simple closed curve of length $L$.
For all $L'>L$ and $\delta>0$ there is a $C^{\infty}$ simple closed curve
$\mathcal{C}'\subset\Sph$ of length $L'$ with
$d_{H}(\mathcal{C},\mathcal{C}')<\delta$.
\end{lemma}

\begin{proof}
Fix $\delta_{0}$ strictly less than both $\delta$ and the normal injectivity
radius of $\mathcal{C}$, so that every normal graph of height at most
$\delta_{0}$ is embedded, and let $\mathcal{C}_{t}$ be the graph of $t\psi$ for
$\psi\colon\mathcal{C}\to[0,1]$ smooth and $0\le t\le\delta_{0}$. Then
$\length(\mathcal{C}_{t})$ is continuous in $t$ and equals $L$ at $t=0$, while a
graph is at least as long as the total variation of its height, so if $\psi$
oscillates $m$ times between $0$ and $1$ then
$\length(\mathcal{C}_{\delta_{0}})\ge2m\delta_{0}$, which exceeds $L'$ for $m$
large. Now apply the intermediate value theorem.
\end{proof}

\begin{proposition}\label{prop:approx}
Let $\mathcal{C}\subset\Sph$ be a rectifiable simple closed curve of length $L$,
with complementary components $\Omega_{1},\Omega_{2}$ of areas $\alpha_{i}$ and
inradii $r_{i}$. There are $C^{\infty}$ simple closed curves $\mathcal{C}_{j}$
and $\delta_{j}\downarrow0$ with $d_{H}(\mathcal{C}_{j},\mathcal{C})\le\delta_{j}$
and $\length(\mathcal{C}_{j})\le L+\delta_{j}$, whose complementary components
$\Omega_{1}^{j},\Omega_{2}^{j}$ can be labeled so that
\begin{equation}\label{eq:sides}
  \Omega_{i}\cap\{\dist(\,\cdot\,,\mathcal{C})>\delta_{j}\}\ \subseteq\ \Omega_{i}^{j},
  \qquad i=1,2 .
\end{equation}
Their areas and inradii then satisfy $\alpha_{i}^{j}\to\alpha_{i}$ and
$r_{i}^{j}\ge r_{i}-\delta_{j}$ for large $j$.
\end{proposition}

\begin{proof}
\emph{The approximants.} Fix $\delta\in(0,\pi/2)$. An inscribed geodesic polygon
of a Jordan curve need not be simple, and that is the only thing
\cite{BoedihardjoGeng2015} is needed for: their Theorem~2.2 produces, for a
Jordan curve in a Riemannian manifold and any $\varepsilon>0$, a partition of
its parameter interval of mesh below $\varepsilon$ whose consecutive points are
joined by unique minimizing geodesics and whose piecewise geodesic interpolant
is again a Jordan curve. The mesh is measured in the parameter, so uniform
continuity converts it into a bound on the diameters of the arcs cut off by the
vertices; taking those at most $\delta/2$ gives a simple inscribed geodesic
polygon $\mathcal{P}$ with $d_{H}(\mathcal{P},\mathcal{C})\le\delta/2$ and
$\length(\mathcal{P})\le L$, no side being longer than the arc it replaces.
Rounding the finitely many corners of $\mathcal{P}$ inside pairwise disjoint
disks of radius at most $\delta/2$, each meeting $\mathcal{P}$ only in the two
sides through its center, keeps the curve simple and within $\delta$ of
$\mathcal{C}$, and costs arbitrarily little length if the disks are small
enough. The result is a $C^{\infty}$ simple closed curve of length at most
$L+\delta$; letting $\delta=\delta_{j}\downarrow0$ gives the $\mathcal{C}_{j}$.

Both replacements are local: a side of $\mathcal{P}$ lies with the arc it
interpolates in one ball of radius $\delta_{j}$, and a rounded corner lies with
the corner it replaces in that corner's disk. Parametrizing the two curves so
that corresponding points share such a ball and sliding each point of
$\mathcal{C}$ to its partner along the minimizing geodesic between them is
therefore a homotopy of loops whose track stays in the closed
$\delta_{j}$-neighborhood of $\mathcal{C}$.

\emph{The two sides.} Hausdorff closeness alone does not say which component of
$\Sph\setminus\mathcal{C}_{j}$ lies on which side of $\mathcal{C}$; the track of
the homotopy does. Fix $x_{0}\in\Omega_{1}$, discard the finitely many $j$ with
$\delta_{j}\ge\dist(x_{0},\mathcal{C})$, let $\Omega_{1}^{j}$ be the component
of $\Sph\setminus\mathcal{C}_{j}$ containing $x_{0}$, and identify
$\Sph\setminus\{x_{0}\}$ with $\R^{2}$, so that $\Omega_{1}$ and
$\Omega_{1}^{j}$ become the unbounded components. A point $x\ne x_{0}$ with
$\dist(x,\mathcal{C})>\delta_{j}$ lies off the track and off $\mathcal{C}_{j}$,
so $\mathcal{C}$ and $\mathcal{C}_{j}$ are homotopic in $\R^{2}\setminus\{x\}$
and wind equally about it; a Jordan curve winds $0$ outside and $\pm1$ inside,
so $x\in\Omega_{1}$ if and only if $x\in\Omega_{1}^{j}$. With $x_{0}$ itself
this is \eqref{eq:sides}, for $i=1$ and hence for $i=2$.

\emph{Areas and inradii.} The sets in \eqref{eq:sides} increase to $\Omega_{i}$,
every point of which is at positive distance from $\mathcal{C}$, so
$\liminf_{j}\alpha_{i}^{j}\ge\alpha_{i}$ for $i=1,2$; as both pairs of areas sum
to $4\pi$, the two bounds together force $\alpha_{i}^{j}\to\alpha_{i}$. Finally,
a point $x\in\Omega_{i}$ realizing $r_{i}$ lies in $\Omega_{i}^{j}$ as soon as
$\delta_{j}<r_{i}$, and then
$r_{i}^{j}\ge\dist(x,\mathcal{C}_{j})\ge r_{i}-\delta_{j}$.
\end{proof}

\begin{proof}[Proof of Lemma~\ref{lem:collar}(a) for rectifiable $\mathcal{C}$]
Label the two components so that $\Omega=\Omega_{1}$ (the labeling of
Subsection~\ref{ssec:tube} plays no role here) and apply
Proposition~\ref{prop:approx}. Write $f^{j},\alpha^{j},r^{j}$ for the collar
area, area and inradius of $\Omega_{1}^{j}$, and put
\[
  F^{j}(\theta)=(L+\delta_{j})\sin\theta+(\alpha^{j}-2\pi)(1-\cos\theta).
\]
Because $\length(\mathcal{C}_{j})\le L+\delta_{j}$ and $\sin\theta\ge0$ on
$[0,\pi]$, the comparison function \eqref{eq:F} of $\Omega_{1}^{j}$ is at most
$F^{j}$; since $\mathcal{C}_{j}$ is $C^{\infty}$, the case already proved
gives $f^{j}\le F^{j}$ on $[0,r^{j}]$. Beyond $r^{j}$ the collar is full, so
$f^{j}(\theta)\le F^{j}(\min\{\theta,r^{j}\})$ for every $\theta\ge0$.

Fix $\theta\le r$. A point of the collar $\Omega\cap T_{\theta}(\mathcal{C})$
lying at distance more than $\delta_{j}$ from $\mathcal{C}$ lies in
$\Omega_{1}^{j}$ by \eqref{eq:sides}, and within $\theta+\delta_{j}$ of
$\mathcal{C}_{j}$; hence
\[
  \area\bigl(\Omega\cap T_{\theta}(\mathcal{C})
    \cap\{\dist(\,\cdot\,,\mathcal{C})>\delta_{j}\}\bigr)
  \ \le\ f^{j}(\theta+\delta_{j}) .
\]
The sets on the left increase to $\Omega\cap T_{\theta}(\mathcal{C})$, whose
area is $f(\theta)$, so
$f(\theta)\le\liminf_{j}f^{j}(\theta+\delta_{j})
\le\liminf_{j}F^{j}(\min\{\theta+\delta_{j},r^{j}\})$, and that limit is
$F(\theta)$: the $F^{j}$ converge to $F$ uniformly, since $\delta_{j}\to0$ and
$\alpha^{j}\to\alpha$, while $\min\{\theta+\delta_{j},r^{j}\}\to\theta$, since
$r^{j}\ge r-\delta_{j}\ge\theta-\delta_{j}$. Thus $f\le F$ on $[0,r]$.

If moreover $\alpha\ge2\pi$ and $r\le\theta\le\pi/2$, then $F'\ge0$ on
$[0,\pi/2]$, so $f(\theta)=\alpha=f(r)\le F(r)\le F(\theta)$.
\end{proof}

\section*{Acknowledgments}

The author thanks Anthropic's Claude Opus 5 for its assistance throughout this
work, particularly with the references and proofs in the appendices.

\bibliographystyle{plain}
\bibliography{refs}

\end{document}